\documentclass[11pt]{article}
\usepackage{cite}
\usepackage[cmex10]{amsmath}
\usepackage{algorithmic}
\usepackage{array}
\usepackage[font=footnotesize]{subfig}
\usepackage{url}
\usepackage{fullpage}
\usepackage{epsfig}
\usepackage{bbm}
\usepackage{enumerate}
\usepackage{amssymb}
\usepackage{amsfonts}
\usepackage{psfrag}
\usepackage{pstricks}
\usepackage{listings}
\usepackage{cite}
\usepackage{subfig}
\usepackage{multirow}

\newcommand{\norm}[1]{\left\lVert\,#1\,\right\rVert}
\newcommand{\vb}[1]{\mathbf{#1}}

\newcommand{\R}{\mathbbm{R}}

\newcommand{\T}{^T}

\renewcommand{\P}{\ensuremath{\operatorname{P}}}
\renewcommand{\Pr}[1]{\P \left \{ #1 \right \}}

\newcommand{\E}{\ensuremath{\operatorname{E}}}
\newcommand{\Ex}[2][]{\ensuremath{\E_{#1} \left[ #2 \right]}}
\newcommand{\var}[1]{\ensuremath{\operatorname{Var}\left( #1\right)}}

\newcommand{\pmat}[1]{\begin{pmatrix}#1\end{pmatrix}}

\newcommand{\bmat}[1]{\begin{bmatrix}#1\end{bmatrix}}

\newcommand{\bigo}[1]{\mathcal{O} \left( #1 \right)}

\newcommand{\minimize}{\text{minimize}}
\newcommand{\maximize}{\text{maximize}}
\newcommand{\st}{\text{subject to}}

\newenvironment{proof}{\textbf{Proof:}}{$\Box$}

\newtheorem{definition}{Definition}

\newtheorem{lemma}{Lemma}
\newtheorem{proposition}{Proposition}
\newtheorem{theorem}{Theorem}

\title{Conditions for Correct Sensor Network Localization \\ Using SDP Relaxation}

\author{Davood Shamsi%
  \thanks{Department of Management Science and Engineering,
          Stanford University, Stanford, CA 94305.
          Email: {\tt davood@stanford.edu}, , {\tt yinyu-ye@stanford.edu}}
  \and Nicole Taheri\thanks{Institute for Computational and Mathematical
       Engineering, Stanford University, Stanford, CA 94305.
       Email: {\tt ntaheri@stanford.edu}}
    \and Zhisu Zhu\thanks{Oracle Corporation, 500 Oracle Parkway, Redwood Shores, California 94065.
       Email: {\tt zhuzhisu@gmail.com}}
  \and Yinyu Ye\footnotemark[1] }

\begin{document}

\maketitle

\begin{abstract}
	
	A Semidefinite Programming (SDP) relaxation is an effective
	computational method to solve a Sensor Network Localization
	problem, which attempts to determine the locations of a group of
	sensors given the distances between some of them.  In this paper, we
	analyze and determine new sufficient conditions and formulations
	that guarantee that the SDP relaxation is exact, i.e., gives the
	correct solution.  These conditions can be useful for designing
	sensor networks and managing connectivities in practice.
	
	Our main contribution is threefold: First, we present the first
	non-asymptotic bound on the connectivity (or radio) range
	requirement of randomly distributed sensors in order to ensure the
	network is uniquely localizable with high probability.  Determining
	this range is a key component in the design of sensor networks, and
	we provide a result that leads to a correct localization of each
	sensor, for any number of sensors.  Second, we introduce a new class
	of graphs that can always be correctly localized by an SDP
	relaxation.  Specifically, we show that adding a simple objective
	function to the SDP relaxation model will ensure that the solution
	is correct when applied to a triangulation graph.  Since
	triangulation graphs are very sparse, this is informationally
	efficient, requiring an almost minimal amount of distance
	information.  Finally, we analyze a number of objective functions for
	the SDP relaxation to solve the localization problem for a general
	graph.

\end{abstract}

\section{Introduction}

\label{section::intro}
Graph Realization is a commonly studied topic which attempts to map
the nodes in a graph $G(V,E)$ to point locations in Euclidean
space based on the non-negative weights of the edges in $E$; that is,
the weight of each edge corresponds to the Euclidean distance between
the incident points.  There are a number of applications of the graph
realization problem \cite{belk07, bulu00, CRIP88, LI08, BGJ09}. In
this paper, we focus on the application to Sensor Network Localization
(SNL).

A sensor network consists of a collection of \emph{sensors} whose
locations are unknown, and \emph{anchors} whose locations are known.
A common property of a sensor network is that each sensor detects
others within a given connectivity (or radio) range and determines the
distance from itself to these nearby sensors.  Given this set of known
distances, the goal is to determine the exact location of each
sensor.  The problem becomes a graph realization problem by forming
the weighted undirected graph $G(V,E)$, where the node set $V$
represents the sensors and each non-negative weighted edge in $E$
represents a known distance between two sensors.

The SNL problem has received a lot of attention recently because of
the formulation of its relaxation as a Semidefinite Program (SDP)
\cite{biswas, manchoso, alfakih97,KW10}.  This formulation can find
the exact locations of the sensors, given that the graph possesses
certain properties.

\begin{definition}
	A \emph{correct localization}, or a \emph{correct solution},
	provides a set of points that is exactly equal to the sensor
	locations.  That is, the solution not only solves a given
	formulation, but it provides the correct sensor locations in the
	desired dimension.
\end{definition}

In this paper, we present a number of additional sufficient conditions
that guarantee unique localizability (and hence a correct
localization) of the SDP relaxation of the SNL problem.  These
conditions can be useful for designing sensor networks and managing
connectivities in practice.

\subsection{Background}
\label{section::background}

We are given a graph $G(V,E\cup \bar{E})$ in a fixed dimension $d$,
where the nodes, or points, of $V$ are partitioned into two sets: the
set $V_a=\{a_1, \ldots, a_m\}$ of $m$ anchors (where $m \geq d + 1$)
whose locations are known and the set $V_x=\{x_1, \ldots, x_n\}$ of
$n$ sensors, whose locations are unknown. The edge set also consists
of two distinct sets: the set $E = \{(i,j): i,j \in V_x\}$ of edges
between sensors, and the set $\bar{E}=\{(k,j): k \in V_a, j \in V_x\}$
of edges between an anchor and a sensor. Moreover, for each $(i,j)\in
E$ (or $(k,j)\in \bar{E}$) the Euclidean distance between sensor $i$ and
sensor $j$ (respectively, anchor $k$ and sensor $j$) is known as $d_{ij}$
(respectively $\bar{d}_{kj}$). The problem of finding the locations of the
sensors can be formulated as finding points $x_1, x_2, \ldots, x_n \in
\R^d$ that satisfy a set of quadratic equations:
\begin{align}
	 \norm{x_i - x_j}^2 = d_{ij}^2,& \ \forall \; (i,j) \in E \notag \\
	 \norm{a_k - x_j}^2 = \bar{d}_{kj}^2,& \ \forall \; (k,j) \in
	 \bar{E}.
	 \tag{SNL-norm}
   \label{eq::orig-form}
\end{align}
From this, a number of fundamental questions naturally arise:  Is
there a localization or realization of $x_j$'s that solves this
system?  If there is a solution, is it unique? And is there a way to
certify that a solution is unique?  Is the network instance partially
localizable, i.e., is the localization solution for a subset of the
sensors unique?  These questions were extensively studied in the graph
rigidity and discrete geometry communities from a more combinatorial
and theoretical prospective (see \cite{hendrickson, connelly} and
references therein). However, the question of whether there is an
efficient algorithm to numerically answer some of these questions
remains open.

The SDP relaxation model \eqref{eq::sdp-form} and corresponding method
aim to answer these questions computationally (see \cite{biswas,
manchoso}).  Let $e_i\in \R^n$ represent the $i$th column of the
identity matrix in $\R^{n \times n}$, and define the symmetric
matrices $A_{ij} := ( \vb 0; e_i - e_j)(\vb 0; e_i - e_j)\T$ and
$\bar{A}_{kj} := (a_k; - e_j) (a_k; - e_j)\T$, where $\vb 0\in\R^d$ is
the vector of all zeros. The SDP relaxation can be represented as:
\begin{equation}
   \begin{array}{lll}
       \maximize & 0 \\
       \st & Z_{(1:d,1:d)} = I_d \\
        & A_{ij} \bullet Z = d_{ij}^2, & \forall (i,j) \in E \\
				&\bar{A}_{kj} \bullet Z = \bar{d}_{kj}^2, & \forall (k,j) \in
					\bar{E} \\
       & Z \succeq 0.
   \end{array}
	 \tag{SNL-SDP}
   \label{eq::sdp-form}
\end{equation}
Here, $Z_{(1:d,1:d)}$ represents the upper-left $d$-dimensional
principle submatrix of $Z$, the matrix dot-product refers to the sum
of element-wise products $A\bullet B=\sum_{ij}A_{ij}B_{ij}$, and $Z
\succeq 0$ means that the symmetric variable matrix $Z$ is positive
semidefinite. Note that problem \eqref{eq::sdp-form} is a convex
semidefinite program and can be approximately solved in polynomial
time by interior-point algorithms.

One can see that the solution matrix $Z \in \R^{(d+n)\times (d+n)}$ of
\eqref{eq::sdp-form} is a matrix that can be decomposed into submatrices,
\[
   Z = \bmat{ I & X \\ X\T & Y }.
\]
The constraint $Z \succeq 0$ holds if and only if $Y \succeq X\T X$. If $Y =
X\T X$, then the above formulation finds a matrix $Z$ such that the
columns of its submatrix ${X = \bmat{x_1 & x_2 & \cdots & x_n}}$
satisfy all quadratic equations in \eqref{eq::orig-form}.

\begin{definition}
A sensor network is \emph{uniquely localizable} if there is a unique
$X \in \R^{d \times n}$ whose columns satisfy \eqref{eq::orig-form},
and there is no $\bar{X} \in \R^{h \times n}$, for $h > d$, whose
columns satisfy $\eqref{eq::orig-form}$ and $\bar{X} \neq (X; \vb 0)$.
In other words, there is no nontrivial extension of $X \in \R^{d
\times n}$ into higher dimension $h > d$ that also satisfies
\eqref{eq::orig-form} \cite{manchoso}.
\end{definition}

Note that the notion of unique localizability is stronger than the
notion of global rigidity.  A sensor network is globally rigid only if
there is a unique $X \in \R^{d \times n}$ that satisfies
\eqref{eq::orig-form}, but it may also have a solution in a higher
dimension space, that is a nontrivial extension of $X \in \R^{d
\times n}$, which satisfies \eqref{eq::orig-form} \cite{manchoso,
AL07}.

The following theorem was proved in \cite{manchoso}:
\begin{theorem}
An SNL problem instance is uniquely localizable if and only if the
maximum rank solution of its SDP relaxation \eqref{eq::sdp-form} has rank
$d$, or equivalently, every solution matrix $Z$ of \eqref{eq::sdp-form}
satisfies $Y = X\T X$. Moreover, such a max-rank solution matrix can
be computed approximately in polynomial time.
\end{theorem}
The theorem asserts that the certification of a uniquely localizable
network instance can be achieved by solving a convex optimization
problem; the proof is constructive and produces a unique realization
or localization solution for the original problem
\eqref{eq::orig-form}.

The dual of the SDP relaxation \eqref{eq::sdp-form}
\begin{align}
       \minimize \; \; & I_d \bullet V + \sum_{(i,j) \in E} y_{ij} d_{ij}^2
           + \sum_{(k,j) \in \bar{E}} w_{kj} \bar{d}_{kj}^2 \notag \\
       \st \; \; & \pmat{ V & 0 \\ 0 & 0 }
           + \sum_{(i,j) \in E} y_{ij} A_{ij}
       + \sum_{(k,j) \in \bar{E}} w_{kj} \bar{A}_{kj} \succeq 0
			 \tag{SDP-dual}
			 \label{eq::sdp-dual}
\end{align}
is also useful, in that the solution to the dual tells us key
properties about the primal. We define the dual slack matrix $U \in
\R^{(d+n)\times(d+n)}$ as
\begin{align*}
U = \pmat{ V & 0 \\ 0 & 0 } + \sum_{(i,j) \in E} y_{ij} A_{ij} +
 \sum_{(k,j) \in \bar{E}} w_{kj} \bar{A}_{kj}.
\end{align*}
for $V\in \R^{d \times d}$.  The dual slack matrix $U$ is optimal if
and only if it is feasible and meets the complementarity condition,
$ZU = 0$.  If complementarity holds, then rank$(Z) + $ rank$(U) \leq
(d+n)$, and since rank$(Z) \geq d$, this means that rank$(U) \leq n$.
Thus, if an optimal dual slack matrix has rank $n$, then every
solution to \eqref{eq::sdp-form} has rank $d$ \cite{manchoso}. In fact, we
have a stronger notion on localizability:
\begin{definition}
A sensor network is \emph{strongly localizable} if there exists an
optimal dual slack matrix with rank $n$.
\end{definition}
Again, such a max-rank dual solution matrix can be computed
approximately in polynomial time using SDP interior-point algorithms.

\subsection{Our Contributions}

In this paper, we present new conditions that guarantee unique
localizability of the SDP relaxation of the problem, i.e., conditions
that ensure the SDP will give the correct solution so that the sensor
network can be localized in polynomial time. We also enhance the
relaxation such that the new SDP relaxation will produce a correct
solution in dimension $d$ to satisfy \eqref{eq::orig-form}, even when
the standard SDP relaxation \eqref{eq::sdp-form} may not.  More precisely,
our result is twofold:

\begin{enumerate}
	 \item A very popular graph in the context of sensor network
		localization is the unit-disk graph, where any two sensor points
		(or a sensor point and an anchor point) are connected if and only
		if their Euclidean distance is less than a given connectivity radius
		$r$.  It has been observed that when the radius (or radio
		range) increases, more sensors in the network can be correctly
		localized.  There is an asymptotic analysis to explain this
		phenomenon when the sensor points are uniformly distributed in a
		unit-square \cite{angluin-stably}.  In this paper, we present a
		\emph{non-asymptotic} bound on the radius requirement of the
		points in order to ensure the network is uniquely localizable with
		high probability.  Specifically, we decompose the area into
		sub-regions, which allows us to analyze whether the locations of
		points in each sub-region can be determined, as opposed to
		analyzing each point individually.  We then determine the
		probability that the locations of all sensors can be
		determined, given a specified concentration of the sensors in a
		given area. This may have practical impact by providing
		guidance on communication power ranges that ensure the network is
		uniquely localizable.

	 \item The basic SDP localization model \eqref{eq::sdp-form} is an SDP
		feasibility problem.  An open question has been to determine
		whether adding a certain objective function to the basic model
		improves localizability of the problem; that is, if the SDP
		feasible region contains high-rank solutions, is the SDP optimal
		solution guaranteed to be unique and low-rank with a certain
		objective?  We give an affirmative answer for a generic class of
		graphs, by identifying an objective function that will always
		result in a correct localization for this class of graphs.  Our
		result may also have an influence on Compressed Sensing, which
		uses an objective function to produce the sparsest solution.
		Based on this idea, we present numerical results by comparing
		several SDP objective functions to illustrate their effectiveness.
\end{enumerate}

Moreover, although our theoretical analyses are based on exact
distance measurements, similar extensions of our model (established in
earlier SDP work) would be applicable to noisy distance data.

\subsection{Paper Organization}
\label{section::organization}

The organization of this paper is as follows. First, Section
\ref{section::bound} derives a lower bound for the connectivity radius
in a sensor network that guarantees unique localizability with high
probability.  In Section \ref{section::triangulation}, we prove that
given a triangulation (i.e., a planar, chordal and convex) graph, if
the sum of the distances between nodes that do \emph{not} have an edge
between them is maximized, then the graph will be strongly
localizable.  We use this idea, and test a number of heuristic
objective functions on a large number of random sensor networks to
determine how well each works in practice.  Our results for these
heuristics are presented in Section \ref{section::obj-fn}.

\section{Bounding the Connectivity Radius}
\label{section::bound}

In this section, we consider the unit-disk graph model
\cite{aspnes-complex, BDHI06, CCJ91} for sensor networks, where the
Euclidean distance between any two sensor points (or a sensor point
and an anchor point) is known (i.e., the two points are connected) if
and only if the distance between them is less than a given
connectivity radius $r$.  Assuming that the sensor points are randomly
distributed in a region,
we then establish a lower bound on radius $r$ that guarantees unique
localizability, with high probability, of the sensor network formed
based on radius $r$.  We do this by establishing a lower bound on
radius $r$ to ensure that the unit-disk graph is a $(d+1)$-lateration
graph, which is a sufficient condition for unique localizability.

\begin{definition}
	 For some $d, n \geq 1$, the graph $G(V,E)$ is a
	 \emph{($d$+1)-lateration graph} if there exists a permutation of
	 the points, $\{\pi(1), \pi(2), \ldots, \pi(n)\}$, such that the
	 edges of the sub-graph ${\pi(1), \ldots, \pi(d+1)}$  form a
	 complete graph, and each successive point $\pi(j)$ for $j \geq d+2$
	 is connected to $d+1$ points in the set $\{\pi(1), \ldots,
	 \pi(j-1)\}$.  This permutation of the points, $\pi$, is called a
	 \emph{$(d+1)$-lateration ordering}.
   \label{def::lateration}
\end{definition}

It is shown in \cite{zhu} that if a sensor network graph contains a
spanning $(d+1)$-lateration graph and the points are in general
position, then it is uniquely localizable.  Zhu et al.\ \cite{zhu}
provide a rigorous proof, which is based on the intuitive concept that
given $d+1$ points in general position forming a complete graph, the
locations of the points can be always be uniquely determined, and the
location of any point connected to $d+1$ points with known locations
can also be determined.

Define $r(p)$ to be the smallest connectivity radius of the randomly
distributed sensor points that ensures the network is uniquely
localizable with probability at least $p$.  To find a lower bound on
$r(p)$, we can find a connectivity radius for which the unit-disk
graph $G(V,E)$ will contain a spanning $(d+1)$-lateration graph with
at least probability $p$.

We approach the problem by considering a unit hypercube $\mathcal{H}
=[0,1]^d$, which contains all the sensor points.  We then split the
region $\mathcal{H}$ into a grid of $M$ equal sub-hypercubes in
dimension $d$, say $h_1, h_2, \ldots, h_M \subset \mathcal{H}$, where
each sub-hypercube $h_i$ will have a volume of $1/M$, and the length
of each of its edges will be $\ell := 1/\sqrt[d]{M}$.  Without loss of
generality, we can assume $M = b^d$, where $b$ is a positive integer
and $b \geq 3$. Similarly, if the region considered is a
hyper-rectangle in dimension $d$, we can assume $M = b_1 \cdot b_2
\cdots b_d$, where $b_i \geq 3$ for $i = 1, \ldots, d$ are positive
integers.  This partition will allow us to analyze the probability
that the locations of sensors in a given region can be determined, as
opposed to analyzing each individual point.

\subsection{Ensuring a Clique in the Graph}
\label{subsection::clique}

Since a $(d+1)$-lateration ordering on the points must begin with a
$(d+1)$-clique, we first find a lower bound on the radius $r$ to
ensure there exists at least one clique of $d+1$ points in the graph.

\begin{proposition}
	Let $\mathcal{H}$ contain $n$ points, and $r \geq
	\ell\sqrt{d}=\frac{ \sqrt{d}}{\sqrt[d]{M}}$ and $M \leq
	\frac{n-1}{d}$ (or equivalently $r \geq \frac{ \sqrt[d]{d}
	\sqrt{d}}{\sqrt[d]{n-1}}$). Then, there exists at least one clique
	of $d+1$ points in the unit-disk graph $G(V,E)$.
	\label{prop::lower-bd}
\end{proposition}

\begin{proof}{
		Note that $\frac{ \sqrt{d}}{\sqrt[d]{M}}$ is the length of the
		diagonal of each sub-hypercube $h_i$. Thus, if $r$ is
		lower-bounded by the given value, then every point in a
		sub-hypercube will be connected to any other point in the same
		sub-hypercube.  Furthermore, since there are at most
		$\frac{n-1}{d}$ sub-hypercubes, by the pigeon-hole principle, at
		least one of them contains at least $(d+1)$ points and they must
		form a clique of $d+1$ points in the unit-disk graph with given
		radius $r$.  } \end{proof}
\vskip 0.1in

In what follows, we fix $n=d\cdot M+1=d\cdot b^d+1$. We will
initialize the spanning $(d+1)$-lateration graph construction by
choosing $r$ according to this lower bound, and let the points in the
$(d+1)$-clique be the first $d+1$ points in the lateration ordering.
Since these points are randomly distributed, they must be in general
position with probability one. Thus, we may assume that these $d+1$
points are anchors for the sensor network. This assumption is without
loss of generality, because our bound on the radius $r$ established in
the following sections will be much greater than the bound specified
in Proposition \ref{prop::lower-bd}, simply because we need to ensure that not
only does there exist a clique of $d+1$ points, but also all sensor
points in $\mathcal{H}$ form a spanning $(d+1)$-lateration graph with
a high probability.

\subsection{Binomial Distribution Model}
\label{section::binomial}

One way to let the sensor points be randomly distributed throughout
the area of $\mathcal{H}$ is to let the points be binomially
distributed throughout
each sub-hypercube of $\mathcal{H}$.
More specifically, the number of points, $Y_i$, placed in each
sub-hypercube $h_i$, for $i=1,...,M$, will be independently and
binomially generated according to ${Y_i \sim B\left( n,\frac{1}{M}
\right)}$ with $n=d\cdot M+1=d\cdot b^d+1$.  Once $Y_i$ is generated,
we let these $Y_i$ sensor points be arbitrarily placed in general
position within sub-hypercube $h_i$.

Using this binomial distribution model, let $S_n = \sum_{i=1}^{M} Y_i$
denote the total number of points in the hypercube $\mathcal{H}$.
Since the $Y_i$ values  are independently and identically distributed and
all sub-hypercubes are equally sized, the total number of points will
be more or less evenly distributed in the entire hypercube
$\mathcal{H}$. Furthermore, by properties of the binomial
distribution,
\begin{align*}
 \Ex{S_n} &= M \cdot \Ex{ Y_1 } =
       M \left( \frac{n}{M}\right ) = n \label{eq::exp-sn} \\
 \var{S_n} &= M\cdot \var{Y_1} =  M
       \cdot \left[\frac{n}{M} \left(1-\frac{1}{M}\right)\right]
     = n\left(1-\frac{1}{M}\right).
\end{align*}
Thus, $\frac{S_n}{n}\to 1$ almost surely and the assumption of
binomially distributed sensor points throughout each sub-hypercube is
statistically equivalent to assuming a uniform distribution of $n$
points throughout the whole region $\mathcal{H}$ when $M$ is
sufficiently large.

\subsection{Connectivity Bound}
\label{subsection::bound}

We now form further conditions on the connectivity radius $r$ to
ensure that the unit-disk graph $G$ contains a spanning
$(d+1)$-lateration graph.  We have assumed that the points are
binomially distributed in each sub-hypercube, parametrized as
${B\left(n,\frac{1}{M} \right)}$.  First, $r$ must satisfy Proposition
\ref{prop::lower-bd}, since it ensures a ($d+1$)-clique in $G$.  These points in
the clique will represent the first $d+1$ points in the lateration
ordering $\pi$ of a spanning $(d+1)$-lateration graph (Definition
\ref{def::lateration}).

We construct an improved bound on the probability of localizability
through an ordering of the hypercubes, $h_i \in \mathcal{H}$, and
hence an ordering on the points.  For simplicity, we prove the
following lemmas for the case of $d=2$, and we refer to the
sub-hypercubes as sub-squares.  We also refer to $(d+1)$-lateration
when $d=2$ as \emph{trilateration}. However, we note that the same
analysis can be applied to hypercubes in higher dimensions, and our
bound $r \geq 2 \ell \sqrt{2}$ in Lemmas
\ref{lemma::sequential}--\ref{lemma::empty} is analogous to the bound $r \geq 2
\ell \sqrt{d}$ in dimension $d$.

\begin{lemma}
	Assume that each sub-square in $\mathcal{H} \in \R^2$ has at least
	one point, and $r \geq 2 \ell \sqrt{2}$.  If the points of three
	sub-squares in the same row in three consecutive columns are in the
	trilateration ordering, then the points in all sub-squares in those
	three columns are also in the ordering.  Similarly, if the points in
	three consecutive sub-squares in the same column are in the
	trilateration ordering, then the points in all sub-squares in those
	three rows are also in the ordering.
 \label{lemma::sequential}
\end{lemma}
\begin{proof}{
				First, note that the lower bound $r \ge 2 \ell \sqrt{2}$
				ensures that all points in a given sub-square are connected to
				all points in a neighboring sub-square, which share either an
				edge or a point within the given sub-square.

				For ease of explanation, let $(i,j)$ represent the sub-square
				in the $i$-th row and $j$-th column, and consider the case
				that all points in the first three sub-squares in the first
				row of the grid are already in the trilateration. Since all
				points in sub-square $(2,2)$  are within the connectivity
				range of all three points in the first row, these points are
				in the trilateration. Then, all points in sub-square $(2,1)$
				(or $(2,3)$) are within the connectivity range of at least
				three points in sub-square $(1,2), (2,2), (1,1)$ (or $(1,2),
				(2,2), (1,3)$), these points are also in the trilateration.
				Therefore, all points in the first three sub-squares of the
				second row are in the trilateration.

				Similarly, all points in the third row
				of the grid in the first three columns are also in the trilateration.  This pattern continues,
				until all points in the first three columns of the grid are in
				the trilateration.

				A generalization of this shows that if there are three
				sub-squares in the same row and in consecutive columns with
				points in the trilateration, and each sub-square has at least
				one point, then all points in the corresponding columns are
				also in the trilateration.

				An analogous result holds for three sub-squares in the same
				column and in consecutive rows.}
\end{proof}

Lemma \ref{lemma::sequential} states that if there are three consecutive
sub-squares in a row with points in the trilateration, then the
trilateration ordering extends to all squares in the corresponding
columns. This concept is used below in Lemma \ref{lemma::nonempty}, which
analyzes the cases depicted in Figure \ref{fig::lemma1}.

\begin{figure}[!htbp]
\centering
   \subfloat{\includegraphics[scale=1.2]{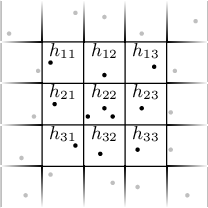}}
    \qquad
   \subfloat{\includegraphics[scale=1.2]{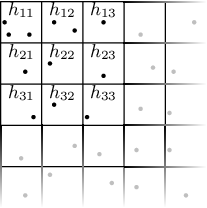}}
 \caption{Conditions as described in Lemma \ref{lemma::nonempty} to ensure
 trilateration}
   \label{fig::lemma1}
\end{figure}

\begin{lemma}{
	Assume there is at least one point in each sub-square and $r \geq
	2\ell\sqrt{2}$.  Then the associated unit-disk graph contains a
	spanning trilateration graph if either:
  \begin{enumerate}[a)]
		\item There is a $3$-clique in a non-corner sub-square
			\label{lemma::cond1}
		\item There is a $3$-clique in a corner sub-square and one of its
			neighbor squares has at least two points\label{lemma::cond2}
	\end{enumerate}
\label{lemma::nonempty} }
\end{lemma}

\begin{proof}{
 Again, note that $r \geq 2\ell\sqrt{2}$ ensures all points in a given
 sub-square are connected to all points in neighboring sub-squares.
 We show that if either of the conditions of Lemma \ref{lemma::nonempty} are
 satisfied, then there exists a trilateration ordering on the points
 in the graph.

\begin{enumerate}[a)]
	\item Consider the example in the left grid of Figure
		\ref{fig::lemma1}, where there is a $3$-clique in the non-corner
		sub-square $h_{22}$.  Let the points in this clique be the initial
		$3$ points in the trilateration ordering.  All points in the
		sub-squares $\{h_{11},h_{12},h_{13},h_{21},h_{23},h_{31},h_{32},
		h_{33}\}$ are connected to this clique; let the points in these
		squares be next in the trilateration ordering.

		By Lemma \ref{lemma::sequential}, all points in the sub-squares in
		rows 1-3 are in the trilateration ordering.  Since there are at
		least three columns in $\mathcal{H}$, the same argument applies
		for the columns, and inductively, there is a trilateration
		ordering on the points that spreads throughout the entire
		hyperspace $\mathcal{H}$.

	\item Now consider the right grid of Figure \ref{fig::lemma1}, where
		there is a $3$-clique in the corner sub-square $h_{11}$, and there
		are at least two points in a neighboring sub-square.  Let the
		points in $h_{11}$ be the first $3$ points in the trilateration
		ordering. All points in the three sub-squares $h_{12}, h_{21},
		h_{22}$ are connected to the points in the clique and hence in the
		trilateration ordering.  Next, let the points in sub-squares
		$h_{31}, h_{32}, h_{33}, h_{23}, h_{13}$ be the succeeding points
		in the ordering.  With a similar argument as before using Lemma
		\ref{lemma::sequential}, we can construct a trilateration on the
		points in the graph, and all points are in the trilateration.
 \end{enumerate}
	Therefore, if the conditions of Lemma \ref{lemma::nonempty} hold, the
	associated unit-disk graph contains a spanning trilateration graph.
} \end{proof} \\

The above lemma provides sufficient, but not necessary, conditions on
a network for trilateration to exist, which implies unique
localizability.  Moreover, these are strict conditions for a sensor
network, since the distribution of sensors in a network may not always
ensure that there is one sensor in each sub-square.  Thus, we extend
these conditions to a more general case, and allow for the possibility
of empty sub-squares.  Clearly, too many empty sub-squares will result
in a graph that is not uniquely localizable; also, if empty
sub-squares exist, there must be restricting conditions to ensure the
graph is not too sparse to ensure localizability.  Thus, we establish
additional properties of the graph that ensure a trilateration but
allow for empty sub-squares.

\begin{definition}
	Two neighboring sub-squares are called \emph{adjacent} neighbors if
	they do not share any edges, but share a point; neighbors that share
	an edge are called \emph{simple} neighbors. A sub-square is called
	\emph{densely surrounded} if all its simple neighbors have at least
	two points and one of its simple neighbors has at least 3 points.
\end{definition}

\begin{figure}[!htbp]
\centering
	 \subfloat[Example of grid that does not satisfy conditions of Lemma
	 \ref{lemma::empty}]{\includegraphics[scale=.7]{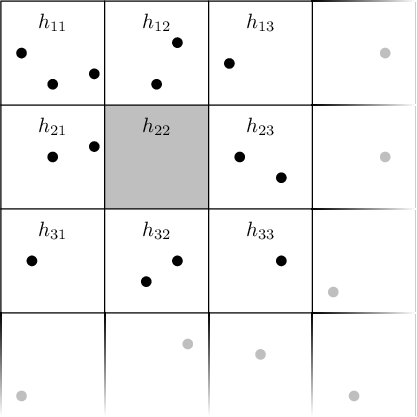}}
    \qquad
	 \subfloat[Example of grid that satisfies conditions of Lemma
	 	\ref{lemma::empty}]{\includegraphics[scale=.7]{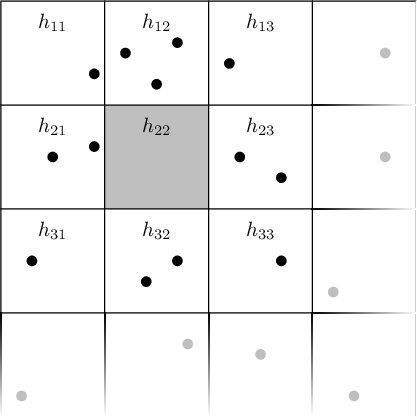}}
	 \caption{Example grids for Lemma \ref{lemma::empty}}
   \label{fig::lemma-empty}
\end{figure}

\begin{lemma}{
	Assume every empty sub-square is densely surrounded and $r \geq
	2\ell\sqrt{2}$.  Then the associated unit-disk graph contains a
	spanning trilateration graph if there is a $3$-clique in a
	non-corner sub-square.
\label{lemma::empty}}
\end{lemma}

\begin{proof}{
	Consider the grids in Figure \ref{fig::lemma-empty}, which shows an example
	and counter example of the conditions in Lemma \ref{lemma::empty}.

	\begin{enumerate}[a)]
	\item First, consider the left grid of Figure \ref{fig::lemma-empty},
		which does not satisfy the condition of Lemma \ref{lemma::empty}
		because there is a clique in a corner sub-square.  Notice that if
		a trilateration ordering starts with the points in sub-square
		$h_{11}$, it can continue to the points in sub-squares $h_{12}$
		and $h_{21}$, but will not spread to points in other sub-squares.
		That is, there is no trilateration ordering that starts with the
		points in $h_{11}$ and extends to the points in the
		sub-squares $\{h_{13}, h_{23}, h_{31}, h_{32}, h_{33}\}$, because
		none of these sub-squares neighbor a subset of sub-squares, in the
		corresponding trilateration ordering, that contain at least 3
		points combined.  Thus, empty sub-squares must be densely
		surrounded to ensure a trilateration ordering on the points
		exists.

	\item Now, consider the right grid of Figure \ref{fig::lemma-empty}, with
		a non-corner $3$-clique, and a densely surrounded empty
		sub-square.  This example shows the worst-case example of the
		condition in Lemma \ref{lemma::empty}.  The shaded sub-square
		$h_{22}$ is empty and densely surrounded, and the 3-clique is
		along the edge of the area $\mathcal{H}$. We prove that a
		trilateration ordering exists on the points in this sample grid
		with a densely surrounded sub-square.  This proves Lemma
		\ref{lemma::empty} holds in the worst-case; the proof that Lemma
		\ref{lemma::empty} holds in every case is a generalized extension
		of this.

		Define the permutation on the points in Figure
		\ref{fig::lemma-empty} via the ordering on their sub-squares:
		\begin{equation}
			\Pi := \{h_{12}, h_{11}, h_{13}, h_{21}, h_{23}, h_{32},
				h_{31}, h_{33}\}.
			\label{eq::empty-order}
		\end{equation}
		This permutation is a trilateration ordering on the points in the
		sample grid. Note that $\Pi$ is a trilateration ordering only
		because the sub-squares $\{h_{21}, h_{23}\}$ contain at least
		three points combined; that is, a permutation containing the
		points in the sub-squares $\{h_{32}, h_{31}, h_{33}\}$ can only be
		a trilateration if $\{h_{21}, h_{23}\}$ together contain at least
		three points.

		By Lemmas \ref{lemma::sequential} and \ref{lemma::nonempty}, if
		there are no other empty sub-squares in $\mathcal{H}$, then there
		is trilateration ordering on all the points in $\mathcal{H}$.
		However, if there \emph{are} other densely surrounded empty
		sub-squares in $\mathcal{H}$, then by a similar construction as
		\eqref{eq::empty-order}, there is still a trilateration ordering
		on all points in $\mathcal{H}$.

	\end{enumerate}

	Therefore, if the condition of Lemma \ref{lemma::empty} holds, the
	associated graph contains a spanning trilateration graph, and hence
	is uniquely localizable in dimension $2$.}
\end{proof}

We now use the fact that a sensor network containing a spanning
trilateration is uniquely localizable \cite{zhu} to establish a lower
bound on the probability that the unit disk sensor network with radius
$r \geq2\ell\sqrt{2}$ is localizable. Define the two events:
\begin{align*}
C &:= \{\text{There are only $3$-cliques in corner sub-squares} \}, \\
\widehat{C} &:= \{\text{There is a $3$-clique in a non-corner
	sub-square} \}.
\end{align*}
Then, the probability that a graph with such randomly distributed
points is uniquely localizable will be
\begin{align*}
\Pr{\text{uniquely localizable}} &= \Pr{\text{uniquely localizable} |
	\widehat{C}} \Pr{\widehat{C}}
 + \Pr{\text{uniquely localizable} |  C}\Pr{C} \\
 &\geq \Pr{\text{uniquely localizable} | \widehat{C}}
 \Pr{\widehat{C}}.
\end{align*}

Given that the total number of sub-squares is $M=b^2$ (for some
integer ${b \ge 3}$), we introduce a parameter $\alpha
:=\sqrt{\frac{n}{M}}$ (or $\alpha :=\sqrt[d]{\frac{n}{M}}$ for general
$d$) such that $\ell = \alpha/\sqrt{n}$ is the edge-length of each
sub-square and we can use the same connectivity radius lower bound as
before, now in terms of $\alpha$, ${r(\alpha) \ge (2 \alpha \sqrt{2})
/\sqrt{n}}$. The distribution of point number in each sub-square is
binomial $B\left( n, \frac{1}{M} \right)$, and there are a total of
$(M-4)$ non-corner sub-squares in $\mathcal{H}$. Thus, the probability
that there is a $3$-clique in a non-corner sub-square is
\[
\Pr{\widehat{C}} = 1 - \left(\sum_{i=0}^2 {n \choose i} \left(
	\frac{1}{M}\right)^i \left(1-\frac{1}{M} \right)^{n-i} \right)^{M-4}.
\]
Let $k$ be the number of empty sub-squares. By Lemma \ref{lemma::nonempty},
${\Pr{\text{uniquely localizable} |k = 0, \widehat{C}} = 1}$, and if
$p_0=(1-\frac{1}{M})^n$ is the probability that one specific
sub-square is empty, we have \[{\Pr{ k = i} = {M \choose i} p_0^i
\left(1-p_0 \right)^{M-i}}.\] Moreover, for any $i<M-4$, we have
\[
\Pr{\widehat{C}|k=i} \geq 1 - \left(\sum_{j=0}^2 {n \choose j}
   \left( \frac{1}{M}\right)^j \left(1-\frac{1}{M} \right)^{n-j} \right)^{M-4-i}
     := p_{\widehat{C},i}.
\]
From Lemma \ref{lemma::empty}, we know
\begin{align*}
	\Pr{\text{uniquely localizable} |k = i, \widehat{C}}
	\geq \Pr{\text{empty sub-squares are densely surrounded} | k = i,
		\widehat{C}}.
\end{align*}
The conditions of Lemma \ref{lemma::empty} require that empty
sub-squares do not have empty simple neighbors; thus, we first find the
probability that a sub-square does not have empty simple neighbors.
Assume there are $k$ empty sub-squares, say $s_1, s_2, \ldots, s_k$.
Because of the independence assumption, these empty sub-squares are
uniformly distributed.

Given the empty sub-square $s_1$, the probability that $s_2$ is not a
simple neighbor of $s_1$ is at least $\left( 1 - \frac{4}{M-1}
\right)$; the probability that $s_3$ is not a simple neighbor of $s_1$
or $s_2$ is at least $(1 - 2\cdot\frac{4}{M-2})$; and so on, so that
the probability that no two empty sub-squares are neighbors is at
least $\prod_{j=1}^{k-1} (1 - \frac{4j}{M-j})$.  Moreover, the
probability that an empty sub-square is densely surrounded, i.e., that
all simple neighbors of an empty sub-square have at least two points
and at least one of them has more than two points, is:
\begin{align*}
\hat{p} &= \Pr{ \text{All simple neighbors have at least two points} }
   -\Pr{ \text{All simple neighbors have exactly two points} } \\
	 &= {\textstyle \left[1 - \sum_{j=0}^1 { n \choose j } \left(
	 	\frac{1}{M}\right)^j \left(1-\frac{1}{M} \right)^{n-j}\right]^4 }
   {\textstyle - \left[{n \choose 2} \left( \frac{1}{M}\right)^2
   \left(1-\frac{1}{M} \right)^{n-2}\right]^4}.
\end{align*}
Thus, the probability that all empty sub-squares are densely
surrounded is
\begin{align*}
\Pr{\text{empty sub-squares are densely surrounded} | \; k = i,
	\widehat{C}} \geq  \hat{p}^i \cdot \prod_{j=1}^{i-1}  (1 -
	\frac{4j}{M-j}).
\end{align*}

Note that the right hand side of the above equation is positive if $i
< M/5$.  Thus, we only consider grids with less than $u := \lfloor
M/5\rfloor - 1$ empty squares.  Finally, we have the lower bound given
by the following expression:
{\small
\begin{align}
   & {\textstyle \Pr{\text{uniquely localizable}}} \notag\\
   & {\textstyle \geq \Pr{\text{uniquely localizable}| \widehat{C}}\Pr{\widehat{C}}}\notag\\
   &{\textstyle = \sum_{i=0}^u \P\{\text{uniquely localizable}
       | k  = i,\widehat{C}\} \P\{\widehat{C}|k=i\} \Pr{k = i} } \notag \\
   &{\textstyle \geq \sum_{i=0}^u \P\{\text{uniquely localizable}
       | k  = i,\widehat{C}\}\Pr{k = i} p_{\widehat{C},i} } \notag \\
   & {\textstyle \geq  p_{\widehat{C},0} \Pr{k = 0}  +
       \sum_{i=1}^u p_{\widehat{C},i}\Pr{k = i}  }
   {\textstyle  \times \Pr{\text{empty sub-squares are
       densely surrounded} | k = i,\widehat{C}} } \notag\\
   & {\textstyle \geq  p_{\widehat{C},0} \Pr{k = 0}  +
       \sum_{i=1}^u  \hat{p}^i \cdot
               p_{\widehat{C},i}\Pr{k = i}  }
    {\textstyle  \times \prod_{j=1}^{i-1}  (1 -
       \frac{4j}{M-j}) }. \label{eq::lowerbd}
\end{align}}

\begin{figure}[t]
   \centering
 \subfloat[$\alpha$ vs.\ number of nodes]
   {\includegraphics[scale=.6]{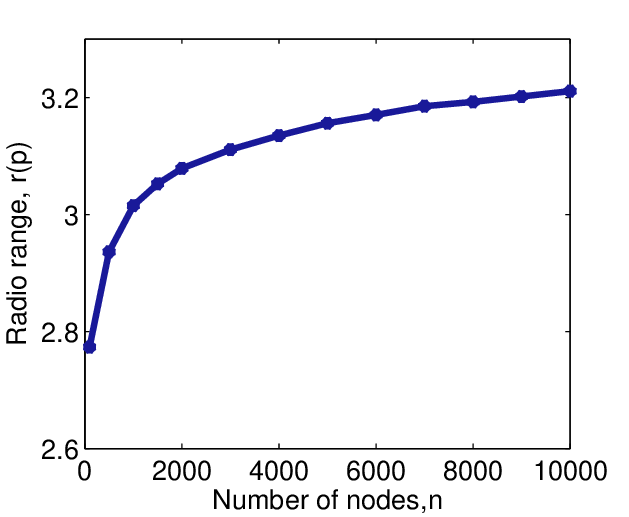}
   \label{fig::alpha}}
\subfloat[$r$ vs.\ number of nodes, compared to Angluin et al.\ bound]
   {\includegraphics[scale=.6]{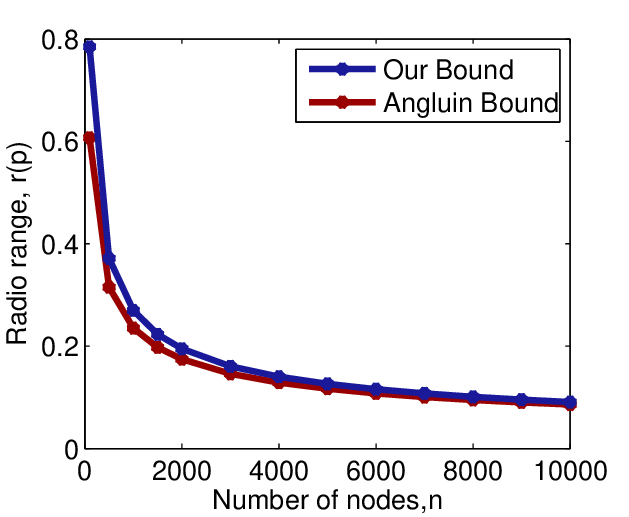}%
   \label{fig::rho}}
	 \begin{center}
   \caption{Bound on the Connectivity Radius}
   \label{fig::range-bound}
   \end{center}
\end{figure}

For different values of $n$ (to be taken as the total number of
sensor points), we can find values of $M$, and thus $\alpha$ (where
$\alpha^2$ can be viewed as the average number of sensor points in
each sub-square), such that the right hand side of Equation
\eqref{eq::lowerbd} is at least $0.99$. Figures \ref{fig::alpha} and
\ref{fig::rho} show $\alpha$ and $r$ versus the number of points $n$
such that the right hand side of Equation \ref{eq::lowerbd} is at
least $0.99$.

We also compare our connectivity bound against the bound of Angulin et
al.\ \cite{angluin-stably} in Figure \ref{fig::range-bound}.  One can
see that our bound and Angluin's are almost identical for any value of
$n$. Thus, our result shows that the bound of Angluin et al. in (of $r
> \frac{2\sqrt{2}\sqrt{\log{n}}}{\sqrt{n}}$ for $d=2$) is true even
when $n$ is small, although it was initially proved to be an
asymptotic bound when $n$ is sufficiently large.  Note that our bound,
while not in an analytical form, is proved for any value of $n$.

We recently learned of another asymptotic bound that was independently
developed by Javanmard and Montanari \cite{javanmard-bound}. However,
this bound is much weaker than ours and Angluin's.

Our connectivity result was proved for $\mathcal{H} =[0,1]^2$, i.e.,
the unit square in dimension $2$.  The result can be extended to
dimension $d > 2$. In summary, we have the following Theorem
\ref{thm::bound}.

\begin{theorem}
Let $\mathcal{H}\in [0,1]^d$ be the unit hypercube in dimension $d$
and be partitioned into a grid of $M=b^d$ equal sub-hypercubes, say
$h_1, h_2, \ldots, h_M \subset \mathcal{H}$, where $\ell=1/b$ is the
edge length of each sub-hypercube.  Let the number of sensor points in
each sub-hypercube be independently and binomially generated according
to ${B\left(n,\frac{1}{M} \right)}$ where $n=d\cdot M+1$, and let one
of the sub-hypercubes contain $d+1$ anchors. Then, if the connectivity
radius satisfies $r \geq 2\ell\sqrt{d}$, the probability that the
sensor network is uniquely localizable is given by expression
\eqref{eq::lowerbd}.
\label{thm::bound}
\end{theorem}

Again, the parameter $n$ of the binomial distribution can be viewed as
the total number of sensor points in the region. We can also extend
our result to another region $\mathcal{H}$ in dimension $d$ into a
grid of $M$ equal sub-hypercubes in dimension $d$, say $h_1, h_2,
\ldots, h_M \subset \mathcal{H}$, where each sub-hypercube $h_i$ will
have a volume of $1/M$, and the length of each of its edges will be
$\ell := 1/\sqrt[d]{M}$.  For example, we can assume $M = b_1 \cdot
b_2 \cdots b_d$, where $b_i \geq 3$ for $i = 1, \ldots, d$ are
positive integers.

\section{Unique Localization of Triangulation Graph}
\label{section::triangulation}

The basic SDP localization model \eqref{eq::sdp-form} is an SDP
feasibility problem. When the network is not uniquely localizable,
the max-rank of SDP feasible solutions is strictly greater than $d$.
In practice, one may still be interested in finding a feasible SDP
solution with rank $d$, representing one possible localization of
points in $\R^d$.  In this section, we show that adding an objective
function that maximizes the sum of certain distances in a
triangulation graph (in $\R^2$) will produce a rank-$2$ SDP solution.
The result should be applicable to $d>2$.

\begin{definition}
	Consider a set of points $\mathcal{P} = \{p_1, p_2 \dots p_n \} \in
	\R^2$. A \emph{triangulation}, $\mathcal{T_P}$, of the points in
	$\mathcal{P}$ is a subdivision of the convex hull of $\mathcal{P}$
	into simplices (triangles) $\{p_i, p_j, p_k\}$, for some $i,j,k \in
	\{1, \dots, n\}$, such that the edges of two simplices do not
	intersect or share a common face.
\end{definition}

\begin{definition}
	For a triangulation $\mathcal{T_P}$, we define a \emph{triangulation
	graph} $G_{\mathcal{T_P}}(V,E)$ such that $V = \mathcal{P}$ and
	${(p_i, p_j) \in E}$ if and only if $(p_i, p_j)$ is an edge of a
	simplex in $\mathcal{T_P}$. Note that triangles in a triangulation graph do
	not overlap, and triangles do not exist strictly inside other triangles.
\end{definition}

Triangulation graphs and their properties have been studied in the
literature \cite{BGJ05, AR05, LCWW03}. Bruck et al. \cite{BGJ05}
showed that embedding a unit disk graph with local angle
information (angles between points) is NP-hard, while the same problem
on a triangulation graph is not.  Ara\'{u}jo and Rodrigues \cite{AR05}
introduced an algorithm to construct a triangulation graph from a unit
disk graph with $\bigo{n\log n}$ bit communications between points.

\begin{figure}[t]
	\begin{center}
\includegraphics[scale=.65]{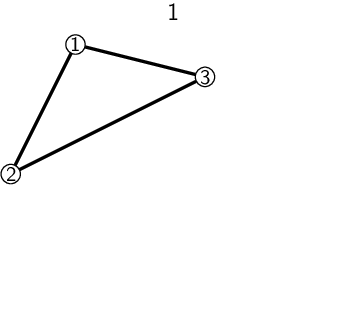}%
\includegraphics[scale=.65]{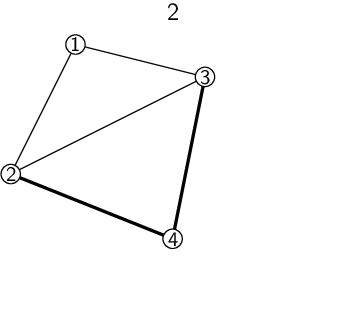}%
\includegraphics[scale=.65]{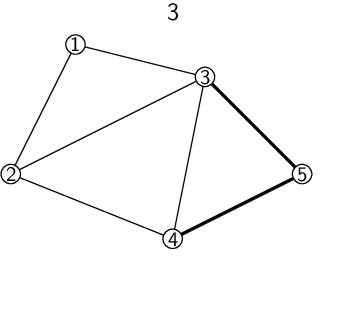}%
\includegraphics[scale=.65]{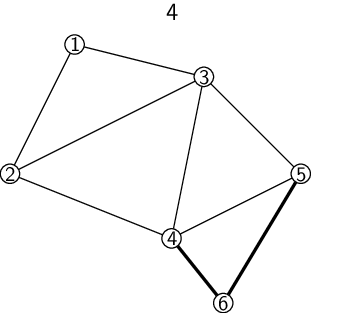}

\includegraphics[scale=.65]{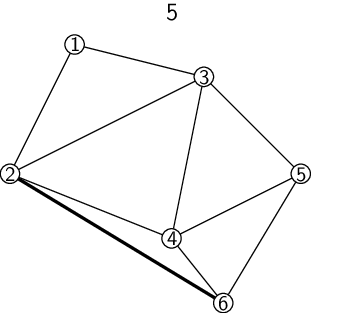}%
\includegraphics[scale=.65]{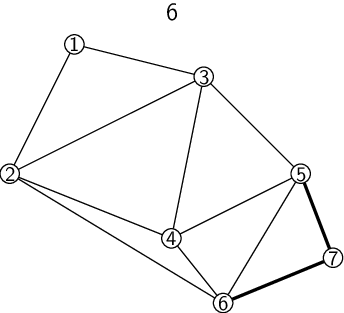}%
\includegraphics[scale=.65]{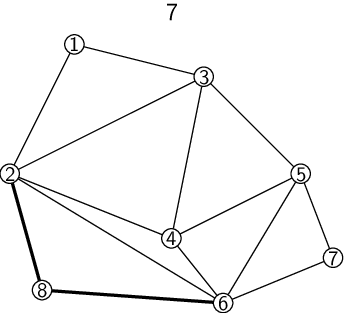}%
\includegraphics[scale=.65]{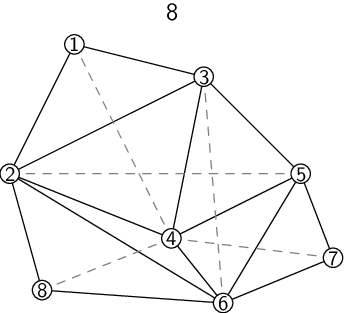}%
	\end{center}
	\caption{Construction of triangulation graph}
	\label{fig::tri-examples}
\end{figure}

We formally decompose a triangulation $\mathcal{T_P}$ into an initial
clique $K_3$ and a set of actions ${\mathcal{A} = \{ a_1,a_2 \dots
a_m\}}$, where an action $a_i$ consists of adding a point and
connecting it to either two adjacent points or two connected external
points, where a point is called \emph{external} if it is not strictly
inside the convex hull of a cycle in the graph.  This leads us to the
following lemma, whose proof is omitted.

\begin{lemma}
	A triangulation can be constructed recursively by either adding an
	external point that connects to two adjacent points of a simplex
	(triangle) already in $\mathcal{T_P}$ such that the new edges do not
	cross any existing edges (see Figure \ref{fig::tri-examples}, 1-4), or
	simply connecting two external points already in $\mathcal{T_P}$ to
	form a triangle (see Figure \ref{fig::tri-examples}, 5).
\end{lemma}
\begin{proof}
   By induction on an external point; see Figure \ref{fig::tri-examples}.
\end{proof}

Step 8 of Figure \ref{fig::tri-examples} shows the set of
\emph{virtual edges} in the sample triangulation graph. These virtual
edges will be used to construct an appropriate objective function of
the SDP relaxation for triangulation graphs.

\begin{definition}
	In a triangulation graph, \emph{adjacent triangles} are two
	triangles which share a common edge.  A \emph{virtual edge} exists
	between two points $i$ and $j$ when $i$ and $j$ belong to adjacent
	triangles, but $(i,j) \not\in E$.  The set of virtual edges between
	sensors is denoted $E_v$, and between sensors and anchors is denoted
	$\bar{E}_v$.
\end{definition}

Consider adding an objective function to the SDP model \eqref{eq::sdp-form}
that maximizes the sum of the lengths of all virtual edges in a generic
triangulation graph. The primal SDP relaxation, for $d=2$, becomes:
\begin{align}
   \maximize \quad & \sum_{(k,j) \in \bar{E}_v} \bar{A}_{kj} \bullet Z
        +  \sum_{(i,j) \in E_v} A_{ij} \bullet Z \notag \\
   \st \quad & Z_{(1:d,1:d)} = I_d \notag \\
       & A_{ij} \bullet Z = d_{ij}^2, \forall (i,j) \in E
			 \label{eq::max-primal}  \\
			 & \bar{A}_{kj} \bullet Z = \bar{d}_{kj}^2, \forall (k,j) \in
			 	\bar{E} \notag \\
       & Z \succeq 0 \notag
\end{align}

and the dual of \eqref{eq::max-primal} is:
\begin{align}
   \minimize \quad & I_d \bullet V + \sum_{(i,j) \in E} y_{ij}d^2_{ij}
       + \sum_{(k,j) \in \bar{E}} w_{kj} \bar{d}^2_{kj} \notag\\
   \st \quad & U = \pmat{ V & 0 \\ 0 & 0}
           + \sum_{(i,j) \in E}y_{ij} A_{ij}
           + \sum_{(k,j) \in \bar{E}} w_{kj} \bar{A}_{kj} \notag \\
       & \quad - \sum_{(k,j) \in \bar{E}_v} \bar{A}_{kj}
           - \sum_{(i,j) \in E_v} A_{ij} \label{eq::max-dual}\\
       & U \succeq 0. \notag
\end{align}


For a triangulation graph with at least three anchors, we can show that
\eqref{eq::max-dual} is strictly feasible, i.e., there exists a feasible $U$
with $U\succeq 0$ (see Proposition 4.1. in \cite{manchoso06}). The primal SDP
\eqref{eq::max-primal} also has a feasible point. As a result, the strong
duality and complementarity condition hold for \eqref{eq::max-primal} and
\eqref{eq::max-dual}.

We derive the following exact-localization theorem.
\begin{theorem}
Consider applying the SDP relaxation \eqref{eq::max-primal} to a generic
triangulation graph with at least three anchors. Then, the rank of an optimal
dual slack matrix of
\eqref{eq::max-dual} is $n$ and the rank of the optimal SDP solution of
\eqref{eq::max-primal} is $d=2$, so that the pair is strictly complementary and
the SDP relaxation produces the correct localization.
\label{thm::obj-fn}
\end{theorem}

\begin{proof}{
We use induction to show that the ranks of the optimal dual slack matrix $U$
and primal SDP solution $Z$ are $n$ and $d=2$, respectively. This implies that
the strict complementarity conditions holds and \eqref{eq::max-primal} produces
the correct localization, that is, the original true positions of the sensor
points of the generic triangulation graph.

Assume the result is true for any triangulation graph with $n$ points.  It
remains to be shown that this also holds for graphs with $n+1$ points.  It is
clearly true for a single simplex when $n = 3$.

Let $X^n \in \R^{d \times n}$ be the correct locations of points,
where the superindex $n$ represents the number of points.  By the
induction assumption, the solution to \eqref{eq::max-primal} is ${Z^n :=
\pmat{ I_d & X^n \\ \left(X^n\right)^T & \left(X^n\right)^TX^n}}$.
Moreover, the optimal dual slack matrix $U^n$ satisfies $U^n \bullet
Z^n = 0$ and has rank $n$; we can write the optimal dual slack matrix
in terms of its submatrices $U^n= \pmat{ U^n_{11} & U^n_{12} \\
U^n_{21} & U^n_{22} }$, where $U^n_{11} \in \R^{d\times d}$ and
$U^n_{22} \in \R^{n \times n}$. Note that $U^n_{22} \succ 0$, which
follows from the fact that rank$(U^n_{22}) = n$ and $U^n \succeq 0$.

The complementarity condition $U^n \bullet Z^n = 0$ means the elements of $U^n$
represent a stress on each edge such that the total force at all non-anchor
points is zero (assuming, without loss of generality, a stress of $-1$ on all
virtual edges).

\begin{definition}
	Given a set of sensor locations $X = [x_1, x_2, \dots, x_n] \in
	\R^{d\times n}$, let $G(V,E \cup \bar{E})$ be the corresponding
	graph.  A matrix $U \in \R^{n\times n}$ is a \emph{stress matrix} of
	the sensor network if it satisfies the constraints of
	\eqref{eq::max-dual} and $U \bullet (X^TX) = 0$.  That is, each
	element of $U$ represents a stress on the associated edge in $E \cup
	\bar{E}$ such that the total force on each non-anchor point is zero.
\end{definition}

We decompose the triangulation graph into an initial simplex $K_3$,
and actions $\mathcal{A} = \{ a_1,a_2 \dots a_m\}$.  Without loss of
generality, we assume the points in the first triangle are anchor
points and let the last points added to the graph be $x_{n+1}$.  For
example, consider Figure \ref{fig::tri-examples}; let $U^7$ be the dual
slack matrix on points 1--7 and assume the subgraph induced on the
first 7 points is uniquely localizable.  When point 8 is added along
with its incident edges, points $(2,4,6,8)$ form a clique (when
including the virtual edge between 4 and 8, which is unique when its
length is maximized). Consider an SDP relaxation problem in dimension 4 that
maximizes the length of the virtual edge between 4 and 8; this problem will
have a unique optimal solution with rank 2 that determines the exact location
of points $(2,4,6,8)$, and an optimal dual slack matrix that forms a stress
matrix for these four points.


Now consider the general case, where $x_{n+1}$ is the last point added
to the graph. A new triangle is created by adding $x_{n+1}$, its
adjacent triangle and the virtual edge, which forms a 4-clique.
Let $\Omega_0$ be the corresponding positive-semidefinite stress matrix on the
graph formed by $x_{n+1}$, the two points adjacent to $x_{n+1}$ (say, $g$ and
$h$) and the point with which $x_{n+1}$ has a virtual edge (say, $k$).  We
examine the case where $g$ and $h$ are sensors, however the case where at least
one of them is an anchor is an easy extension. As before, the locations $x_g,
x_h, x_k$, and $x_{n+1}$ can be uniquely determined by solving an SDP
relaxation, and $\Omega_0$ is the optimal dual slack matrix that solves its
dual problem:
{
\begin{align}
\minimize \quad & y_{gk}d_{gk}^2 +y_{hk}d_{hk}^2+y_{gh}d_{gh}^2+y_{g,n+1}d_{g,n+1}^2+y_{h,n+1}d_{h,n+1}^2 \label{eq::K4max-dual} \\
   \st \quad & U^4 \succeq 0, \notag
\end{align}}
where
{\small
\begin{align*}
U^4=
   \pmat{ - 1 + y_{gk} + y_{hk} & - y_{gk} & -y_{hk} & 1 \\
   -y_{gk} & y_{gk} + y_{gh} + y_{g,n+1} & - y_{gh} & - y_{g,n+1} \\
   -y_{hk} & - y_{gh} & y_{hk} + y_{gh} + y_{h,n+1} & -y_{h,n+1} \\
   1 & -y_{g,n+1} & -y_{h,n+1} & -1 + y_{g,n+1} + y_{h,n+1} }.
\end{align*}}

Assume $(y_{gk}, y_{hk}, y_{gh}, y_{g,n+1}, y_{h,n+1})$ is the optimal solution
of this SDP, then 
\[
\Omega_0 = U^4.
\]
It's easy to see that strong duality and complementarity condition hold for
this SDP, and therefore $$\sum_{i,j \in \{g, h, k, n+1\}} [\Omega_0 ]_{ij} (x_i^T
x_j)=0.$$

Note that ${0 < (-1 +y_{g,n+1} +y_{h,n+1})}$ because $\Omega_0
\succeq 0$,  and consider the updated stress matrix
\[
U_{22}^{n+1} := \pmat{ U_{22}^n & 0_{n\times 1} \\ 0_{1\times n} & 0}
+ \Omega,
\]
where $\Omega \in \R^{(n + 1)\times (n + 1)}$ is the stress matrix of
the new edges, that is, $\Omega_{([g, h, k, n+1],[g, h, k, n+1])} =
\Omega_0$.

The new matrix $U^{n+1}$ will be feasible for the dual, since $\Omega_0$ is the
solution of \eqref{eq::K4max-dual}, and $\Omega \succeq 0$, $U^n \succeq 0$
implies that $U^{n+1} \succeq 0$.

Define
{\small
\[
Z^{n+1} := \pmat{ Z^n & \pmat{x_{n+1} \\ (X^n)^T x_{n+1}} \\
\pmat{ x_{n+1}^T & x_{n+1}^T X_n} & x_{n+1}^T x_{n+1} }
\]}
as the correct locations of the updated points.  (Note that given this
definition of $Z^{n+1}$, it does not immediately follow that
rank$(Z^{n+1}) = d$, since the added last row of $Z^{n+1}$ can be
linearly independent from the first $n$ rows.) The sum of element-wise
products of $U^{n+1}$ and $Z^{n+1}$ is
\begin{align*}
 &U^{n+1} \bullet Z^{n+1} =
   U^{n} \bullet Z^{n} + \sum_{(i,j)} [\Omega_0 ]_{ij} (x_i^T x_j) = 0.
\end{align*}

Moreover, we can show that $U^{n+1}_{22} \succ 0$.  Assume this is not
true, i.e., assume that there is a vector $z \in \R^{n+1}$ such that
\[
   z\T U_{22}^{n+1} z =
       z^T \bmat{U^n_{22} & 0 \\ 0 & 0} z + z^T \Omega z = 0,
\]
which holds if and only if $z^T  \bmat{U^n_{22} & 0 \\ 0 & 0}z = 0$
and $z^T \Omega z = 0$.  Since $U_{22}^n \succ 0$, this means that the
first $n$ elements of $z$ are zero, i.e., $z_{(1:n)} = 0$.  Thus,
\[
z^T \Omega z = z_{n+1}^2 \Omega_{n+1} = z_{n+1}^2 (-1 +y_{g,n+1} +y_{h,n+1}) = 0
\]
which implies $z_{n+1} = 0$. 
Thus, $z\T U_{22}^{n+1} z = 0$ if and only if $z = 0$, implying
$U_{22}^{n+1} \succ 0$ and rank$(U^{n+1}) = n+1$.  Therefore, the rank
of $(Z^{n+1})$ is $d$,  and consequently from \cite{manchoso}, $Z^{n+1}$
is the unique solution to \eqref{eq::max-primal}, so that the
localization is correct and exact.}
\end{proof}

Theorem \ref{thm::obj-fn} implies that the strict complementarity condition
holds when localizing a generic triangulation graph with the selected objective
function.  This result is interesting because, in general, it is difficult to
prove strict complementarity for SDPs. How to compute a stress matrix (or
optimal dual matrix) and determine whether the stress matrix has rank $n$ are
also important questions in rigidity theory for graph realization. Clearly,
Theorem \ref{thm::obj-fn} is applicable to any graph that contains a generic
triangulation graph as a spanning subgraph. In practice, the objective of the
SDP relaxation may include all non-edges that are not specified in the given
graph (rather than just virtual edges), which we experiment in the next
section.


\section{Heuristic Objective Function}
\label{section::obj-fn}

Section \ref{section::triangulation} proves that adding a given
objective function to \eqref{eq::sdp-form} results in a correct
localization for a certain class of graphs, whereas the formulation
without an objective function may not.

Based on these findings, we tested a number of different SDP
relaxation methods with different objective functions. For each method
(i.e., each objective function), we ran the relaxation on a large
number of random sensor networks and determined the success rate of
each method.  The following objective functions were tested to
heuristically determine the best method.

\begin{enumerate}

	\item (ZERO) Solve the formulation \eqref{eq::sdp-form} (with no
		objective function).  This can be viewed as a control simulation
		against which to compare other methods.

	\item (MAX) Maximize the sum of all the `non-edge' lengths by solving
		the formulation:
		\begin{equation}
		   \begin{array}{lll}
		       \maximize & \sum_{(i,j) \not\in E} d_{ij} + \sum_{(k,j)
					 \not\in \bar{E}} \bar{d}_{kj} \\
		       \st & Z_{(1:d,1:d)} = I_d \\
		        & A_{ij} \bullet Z = d_{ij}^2, & \forall (i,j) \in E \\
						&\bar{A}_{kj} \bullet Z = \bar{d}_{kj}^2, & \forall (k,j) \in
							\bar{E} \\
		       & Z \succeq 0.
		   \end{array}
			 \tag{SDP-MAX}
			 \label{eq::sdp-max}
		\end{equation}

	\item (MIN) Minimize the sum of all the `non-edge' lengths by solving
		the formulation:
		\begin{equation}
		   \begin{array}{lll}
		       \minimize & \sum_{(i,j) \not\in E} d_{ij} + \sum_{(k,j)
					 \not\in \bar{E}} \bar{d}_{kj} \\
		       \st & Z_{(1:d,1:d)} = I_d \\
		        & A_{ij} \bullet Z = d_{ij}^2, & \forall (i,j) \in E \\
						&\bar{A}_{kj} \bullet Z = \bar{d}_{kj}^2, & \forall (k,j) \in
							\bar{E} \\
		       & Z \succeq 0.
		   \end{array}
			 \tag{SDP-MIN}
			 \label{eq::sdp-min}
		\end{equation}

	\item (MAX-PT) Maximize the sum of the distances from each sensor
		location $x_i \in \R^d$ to a distant point, where $x_i$ is set to
		the corresponding elements of the decision matrix $Z$.  For
		example, for $\vb 1 \in \R^d$ the vector of all ones, we took the
		point $p := 1000 \cdot \vb 1$ and solved the formulation:
		\begin{equation*}
		   \begin{array}{lll}
		       \maximize & \sum_{i=1}^n \norm{p - x_i}^2 \\
					 \st & Z_{(1:d,1:d)} = I_d \\
		        & A_{ij} \bullet Z = d_{ij}^2, & \forall (i,j) \in E \\
						&\bar{A}_{kj} \bullet Z = \bar{d}_{kj}^2, & \forall (k,j) \in
							\bar{E} \\
		       & Z \succeq 0.
		   \end{array}
		\end{equation*}

\end{enumerate}

We constructed 200 uniformly distributed sensor networks and tested
each method on the networks for a number of different radio ranges.
Each randomly distributed sensor network has 100 points in a unit
square (in dimension $d=2$), and the distance between two points is
known when they are within the given radio range.  Table
\ref{table::results} shows the percent of sensor networks that were
correctly localized using each method, for each radio range.

\begin{table}[t]
	\begin{center}
	\begin{tabular}{cccccc}
		&& \multicolumn{4}{c}{\underline{Method}} \\
		& \multicolumn{1}{c}{} & \multicolumn{1}{c}{ZERO} &
		\multicolumn{1}{c}{MAX} & \multicolumn{1}{c}{MIN} &
		\multicolumn{1}{c}{MAX-PT} \\
		\cline{3-6}
		\multirow{4}{*}{\underline{Radio Range}} &
		\multicolumn{1}{c|}{0.15} & \multicolumn{1}{|c|}{0} &
		\multicolumn{1}{|c|}{0} & \multicolumn{1}{|c|}{0} &
		\multicolumn{1}{|c|}{0}\\
		\cline{3-6}
		 & \multicolumn{1}{c|}{0.2} & \multicolumn{1}{|c|}{41} &
		 \multicolumn{1}{|c|}{75} & \multicolumn{1}{|c|}{39} &
		\multicolumn{1}{|c|}{0}\\
		\cline{3-6}
		& \multicolumn{1}{c|}{0.25} & \multicolumn{1}{|c|}{87} &
		 \multicolumn{1}{|c|}{95} & \multicolumn{1}{|c|}{88} &
		\multicolumn{1}{|c|}{0}\\
		\cline{3-6}
		& \multicolumn{1}{c|}{0.3} & \multicolumn{1}{|c|}{98} &
		 \multicolumn{1}{|c|}{100} & \multicolumn{1}{|c|}{100} &
		\multicolumn{1}{|c|}{4}\\
		\cline{3-6}
		& \multicolumn{1}{c|}{0.35} & \multicolumn{1}{|c|}{100} &
		 \multicolumn{1}{|c|}{100} & \multicolumn{1}{|c|}{100} &
		\multicolumn{1}{|c|}{7}\\
		\cline{3-6}
		& \multicolumn{1}{c|}{0.4} & \multicolumn{1}{|c|}{100} &
		 \multicolumn{1}{|c|}{100} & \multicolumn{1}{|c|}{100} &
		\multicolumn{1}{|c|}{13}\\
		\cline{3-6}
	\end{tabular}
	\end{center}
	\caption{Percent of Networks Correctly Localized}
	\label{table::results}
\end{table}

As can be seen from Table \ref{table::results}, maximizing the sum of
the unknown distances out-performs the other three methods tested, and
maximizing the sum of distances from a distant point does not produce
good results. Moreover, the methods (ZERO), (MAX) and (MIN) all seemed
to work very well when the radio range was at least 0.35. This radio
range is much smaller than the lower bound given by
\eqref{eq::lowerbd}, but has not been theoretically proved as a radio
range that will lead to a correct localization.

\nocite{*}
\bibliographystyle{plain}
\bibliography{main}

\end{document}